\documentclass[12pt]{amsart}
\usepackage[utf8]{inputenc} 
\usepackage[T1]{fontenc}
\usepackage{color}
\usepackage{comment}
\usepackage{times}
\usepackage[hidelinks]{hyperref}
\usepackage{enumerate,latexsym}
\usepackage{latexsym}
\usepackage{amsmath,amssymb}
\usepackage{graphicx}
\usepackage{subcaption}
\usepackage{array}
\usepackage{makecell}
\usepackage{mathtools}

\newcolumntype{M}[1]{>{\centering\arraybackslash}m{#1}}
\usepackage{longtable}

\usepackage{fullpage}

\usepackage{amsmath,amsthm,amsfonts,amssymb}
\usepackage{graphicx}
\usepackage{dsfont}
\usepackage{tkz-berge}
\usepackage{float}

\usepackage{tikz}
\usetikzlibrary{knots}
\newcounter{braid}
\newcounter{strands}
\usetikzlibrary{decorations.pathreplacing}
\pgfkeyssetvalue{/tikz/braid height}{1cm}
\pgfkeyssetvalue{/tikz/braid width}{1cm}
\pgfkeyssetvalue{/tikz/braid start}{(0,0)}
\pgfkeyssetvalue{/tikz/braid colour}{black}
\pgfkeys{/tikz/strands/.code={\setcounter{strands}{#1}}}

\usepackage{tikz}
\usetikzlibrary{
  knots,
  hobby,
  decorations.pathreplacing,
  decorations.markings,
  arrows.meta,
  shapes.geometric,
  calc,
  external,
  fadings
}
\tikzset{
  knot diagram/every strand/.append style={
    ultra thick,
    black
  },
  show curve controls/.style={
    postaction=decorate,
    decoration={show path construction,
      curveto code={
        \draw [blue, dashed]
        (\tikzinputsegmentfirst) -- (\tikzinputsegmentsupporta)
        node [at end, draw, solid, red, inner sep=5pt]{};
        \draw [blue, dashed]
        (\tikzinputsegmentsupportb) -- (\tikzinputsegmentlast)
        node [at start, draw, solid, red, inner sep=5pt]{}
        node [at end, fill, blue, ellipse, inner sep=5pt]{}
        ;
      }
    }
  },
  show curve endpoints/.style={
    postaction=decorate,
    decoration={show path construction,
      curveto code={
        \node [fill, blue, ellipse, inner sep=5pt] at (\tikzinputsegmentlast) {}
        ;
      }
    }
  }
}

\makeatletter
\def\cross{%
  \@ifnextchar^{\message{Got sup}\cross@sup}{\cross@sub}}

\def\cross@sup^#1_#2{\render@cross{#2}{#1}}

\def\cross@sub_#1{\@ifnextchar^{\cross@@sub{#1}}{\render@cross{#1}{1}}}

\def\cross@@sub#1^#2{\render@cross{#1}{#2}}

\def\render@cross#1#2{
  \def\strand{#1}
  \def\crossing{#2}
  \pgfmathsetmacro{\cross@y}{-\value{braid}*\braid@h}
  \pgfmathtruncatemacro{\nextstrand}{#1+1}
  \foreach \thread in {1,...,\value{strands}}
  {
    \pgfmathsetmacro{\strand@x}{\thread * \braid@w}
    \ifnum\thread=\strand
    \pgfmathsetmacro{\over@x}{\strand * \braid@w + .5*(1 - \crossing) * \braid@w}
    \pgfmathsetmacro{\under@x}{\strand * \braid@w + .5*(1 + \crossing) * \braid@w}
    \draw[braid] \pgfkeysvalueof{/tikz/braid start} +(\under@x pt,\cross@y pt) to[out=-90,in=90] +(\over@x pt,\cross@y pt -\braid@h);
    \draw[braid] \pgfkeysvalueof{/tikz/braid start} +(\over@x pt,\cross@y pt) to[out=-90,in=90] +(\under@x pt,\cross@y pt -\braid@h);
    \else
    \ifnum\thread=\nextstrand
    \else
     \draw[braid] \pgfkeysvalueof{/tikz/braid start} ++(\strand@x pt,\cross@y pt) -- ++(0,-\braid@h);
    \fi
   \fi
  }
  \stepcounter{braid}
}

\tikzset{braid/.style={double=\pgfkeysvalueof{/tikz/braid colour},double distance=1pt,line width=2pt,white}}

\newcommand{\braid}[2][]{%
  \begingroup
  \pgfkeys{/tikz/strands=2}
  \tikzset{#1}
  \pgfkeysgetvalue{/tikz/braid width}{\braid@w}
  \pgfkeysgetvalue{/tikz/braid height}{\braid@h}
  \setcounter{braid}{0}
  \let\sigma=\cross
  #2
  \endgroup
}
\makeatother

\def\mOne{\begin{bmatrix}
1 & 0 & 0 & 0 & 0 & 0 & 0 & 0 & 0 \\
0 & 1 & 0 & 0 & 0 & 0 & 0 & 0 & 0 \\
0 & 0 & 1 & 0 & 0 & 0 & 0 & 0 & 0 \\
0 & 0 & 0 & 1 & 0 & 0 & 0 & 0 & 0 \\
0 & 0 & 0 & 0 & 1 & 0 & 0 & 0 & 0 \\
0 & 0 & 0 & 0 & 0 & 1 & 0 & 0 & 0 \\
0 & 0 & 0 & 0 & 0 & 0 & 1 & 0 & 0 \\
0 & 0 & 0 & 0 & 0 & 0 & 0 & 1 & 0 \\
0 & 0 & 0 & 0 & 0 & 0 & 0 & 0 & 1 \\
\end{bmatrix}}

\def\mTwo{\begin{bmatrix}
1 & 0 & 0 & 0 & 0 & 0 & 0 & 0 & 0 \\
0 & 0 & 0 & 0 & 0 & 1 & 0 & 0 & 0 \\
0 & 0 & 0 & 0 & 0 & 0 & 0 & 1 & 0 \\
0 & 0 & 1 & 0 & 0 & 0 & 0 & 0 & 0 \\
0 & 0 & 0 & 0 & 1 & 0 & 0 & 0 & 0 \\
0 & 0 & 0 & 0 & 0 & 0 & 1 & 0 & 0 \\
0 & 1 & 0 & 0 & 0 & 0 & 0 & 0 & 0 \\
0 & 0 & 0 & 1 & 0 & 0 & 0 & 0 & 0 \\
0 & 0 & 0 & 0 & 0 & 0 & 0 & 0 & 1 \\
\end{bmatrix}}

\def\mFour{\begin{bsmallmatrix}
1 & 0 & 0 & 0 & 0 & 0 & 0 & 0 & 0 & 0 & 0 & 0 & 0 & 0 & 0 & 0 & 0 & 0 & 0 & 0 & 0 & 0 & 0 & 0 & 0 & 0 & 0 \\
0 & 0 & 0 & 0 & 0 & 0 & 0 & 0 & 0 & 0 & 0 & 0 & 0 & 0 & 0 & 0 & 0 & 0 & 0 & 0 & 0 & 0 & 0 & 0 & 0 & 1 & 0 \\
0 & 0 & 0 & 0 & 0 & 0 & 0 & 0 & 0 & 0 & 0 & 0 & 0 & 0 & 1 & 0 & 0 & 0 & 0 & 0 & 0 & 0 & 0 & 0 & 0 & 0 & 0 \\
0 & 0 & 0 & 0 & 0 & 0 & 0 & 0 & 0 & 0 & 1 & 0 & 0 & 0 & 0 & 0 & 0 & 0 & 0 & 0 & 0 & 0 & 0 & 0 & 0 & 0 & 0 \\
0 & 0 & 0 & 0 & 0 & 0 & 0 & 0 & 1 & 0 & 0 & 0 & 0 & 0 & 0 & 0 & 0 & 0 & 0 & 0 & 0 & 0 & 0 & 0 & 0 & 0 & 0 \\
0 & 0 & 0 & 0 & 0 & 0 & 0 & 0 & 0 & 0 & 0 & 0 & 0 & 0 & 0 & 0 & 0 & 0 & 0 & 0 & 0 & 1 & 0 & 0 & 0 & 0 & 0 \\
0 & 0 & 0 & 0 & 0 & 0 & 0 & 0 & 0 & 0 & 0 & 0 & 0 & 0 & 0 & 0 & 0 & 0 & 0 & 0 & 1 & 0 & 0 & 0 & 0 & 0 & 0 \\
0 & 0 & 0 & 0 & 0 & 0 & 0 & 0 & 0 & 0 & 0 & 0 & 0 & 0 & 0 & 1 & 0 & 0 & 0 & 0 & 0 & 0 & 0 & 0 & 0 & 0 & 0 \\
0 & 0 & 0 & 0 & 1 & 0 & 0 & 0 & 0 & 0 & 0 & 0 & 0 & 0 & 0 & 0 & 0 & 0 & 0 & 0 & 0 & 0 & 0 & 0 & 0 & 0 & 0 \\
0 & 0 & 0 & 0 & 0 & 0 & 0 & 0 & 0 & 0 & 0 & 0 & 0 & 0 & 0 & 0 & 0 & 1 & 0 & 0 & 0 & 0 & 0 & 0 & 0 & 0 & 0 \\
0 & 0 & 0 & 1 & 0 & 0 & 0 & 0 & 0 & 0 & 0 & 0 & 0 & 0 & 0 & 0 & 0 & 0 & 0 & 0 & 0 & 0 & 0 & 0 & 0 & 0 & 0 \\
0 & 0 & 0 & 0 & 0 & 0 & 0 & 0 & 0 & 0 & 0 & 0 & 0 & 0 & 0 & 0 & 0 & 0 & 0 & 1 & 0 & 0 & 0 & 0 & 0 & 0 & 0 \\
0 & 0 & 0 & 0 & 0 & 0 & 0 & 0 & 0 & 0 & 0 & 0 & 0 & 0 & 0 & 0 & 0 & 0 & 0 & 0 & 0 & 0 & 0 & 0 & 1 & 0 & 0 \\
0 & 0 & 0 & 0 & 0 & 0 & 0 & 0 & 0 & 0 & 0 & 0 & 0 & 1 & 0 & 0 & 0 & 0 & 0 & 0 & 0 & 0 & 0 & 0 & 0 & 0 & 0 \\
0 & 0 & 1 & 0 & 0 & 0 & 0 & 0 & 0 & 0 & 0 & 0 & 0 & 0 & 0 & 0 & 0 & 0 & 0 & 0 & 0 & 0 & 0 & 0 & 0 & 0 & 0 \\
0 & 0 & 0 & 0 & 0 & 0 & 0 & 1 & 0 & 0 & 0 & 0 & 0 & 0 & 0 & 0 & 0 & 0 & 0 & 0 & 0 & 0 & 0 & 0 & 0 & 0 & 0 \\
0 & 0 & 0 & 0 & 0 & 0 & 0 & 0 & 0 & 0 & 0 & 0 & 0 & 0 & 0 & 0 & 0 & 0 & 0 & 0 & 0 & 0 & 0 & 1 & 0 & 0 & 0 \\
0 & 0 & 0 & 0 & 0 & 0 & 0 & 0 & 0 & 1 & 0 & 0 & 0 & 0 & 0 & 0 & 0 & 0 & 0 & 0 & 0 & 0 & 0 & 0 & 0 & 0 & 0 \\
0 & 0 & 0 & 0 & 0 & 0 & 0 & 0 & 0 & 0 & 0 & 0 & 0 & 0 & 0 & 0 & 0 & 0 & 0 & 0 & 0 & 0 & 1 & 0 & 0 & 0 & 0 \\
0 & 0 & 0 & 0 & 0 & 0 & 0 & 0 & 0 & 0 & 0 & 1 & 0 & 0 & 0 & 0 & 0 & 0 & 0 & 0 & 0 & 0 & 0 & 0 & 0 & 0 & 0 \\
0 & 0 & 0 & 0 & 0 & 0 & 1 & 0 & 0 & 0 & 0 & 0 & 0 & 0 & 0 & 0 & 0 & 0 & 0 & 0 & 0 & 0 & 0 & 0 & 0 & 0 & 0 \\
0 & 0 & 0 & 0 & 0 & 1 & 0 & 0 & 0 & 0 & 0 & 0 & 0 & 0 & 0 & 0 & 0 & 0 & 0 & 0 & 0 & 0 & 0 & 0 & 0 & 0 & 0 \\
0 & 0 & 0 & 0 & 0 & 0 & 0 & 0 & 0 & 0 & 0 & 0 & 0 & 0 & 0 & 0 & 0 & 0 & 1 & 0 & 0 & 0 & 0 & 0 & 0 & 0 & 0 \\
0 & 0 & 0 & 0 & 0 & 0 & 0 & 0 & 0 & 0 & 0 & 0 & 0 & 0 & 0 & 0 & 1 & 0 & 0 & 0 & 0 & 0 & 0 & 0 & 0 & 0 & 0 \\
0 & 0 & 0 & 0 & 0 & 0 & 0 & 0 & 0 & 0 & 0 & 0 & 1 & 0 & 0 & 0 & 0 & 0 & 0 & 0 & 0 & 0 & 0 & 0 & 0 & 0 & 0 \\
0 & 1 & 0 & 0 & 0 & 0 & 0 & 0 & 0 & 0 & 0 & 0 & 0 & 0 & 0 & 0 & 0 & 0 & 0 & 0 & 0 & 0 & 0 & 0 & 0 & 0 & 0 \\
0 & 0 & 0 & 0 & 0 & 0 & 0 & 0 & 0 & 0 & 0 & 0 & 0 & 0 & 0 & 0 & 0 & 0 & 0 & 0 & 0 & 0 & 0 & 0 & 0 & 0 & 1 \\
\end{bsmallmatrix}}

\def\mFive{\begin{bsmallmatrix}
1 & 0 & 0 & 0 & 0 & 0 & 0 & 0 & 0 & 0 & 0 & 0 & 0 & 0 & 0 & 0 & 0 & 0 & 0 & 0 & 0 & 0 & 0 & 0 & 0 & 0 & 0 \\
0 & 0 & 0 & 0 & 0 & 0 & 0 & 0 & 0 & 0 & 0 & 0 & 0 & 0 & 0 & 0 & 0 & 0 & 0 & 0 & 1 & 0 & 0 & 0 & 0 & 0 & 0 \\
0 & 0 & 0 & 0 & 0 & 0 & 0 & 0 & 0 & 0 & 1 & 0 & 0 & 0 & 0 & 0 & 0 & 0 & 0 & 0 & 0 & 0 & 0 & 0 & 0 & 0 & 0 \\
0 & 0 & 0 & 0 & 0 & 0 & 0 & 0 & 0 & 0 & 0 & 0 & 0 & 0 & 0 & 1 & 0 & 0 & 0 & 0 & 0 & 0 & 0 & 0 & 0 & 0 & 0 \\
0 & 0 & 0 & 0 & 0 & 0 & 0 & 0 & 1 & 0 & 0 & 0 & 0 & 0 & 0 & 0 & 0 & 0 & 0 & 0 & 0 & 0 & 0 & 0 & 0 & 0 & 0 \\
0 & 0 & 0 & 0 & 0 & 0 & 0 & 0 & 0 & 0 & 0 & 0 & 0 & 0 & 0 & 0 & 0 & 0 & 0 & 0 & 0 & 0 & 0 & 0 & 0 & 1 & 0 \\
0 & 0 & 0 & 0 & 0 & 0 & 0 & 0 & 0 & 0 & 0 & 0 & 0 & 0 & 0 & 0 & 0 & 0 & 0 & 0 & 0 & 1 & 0 & 0 & 0 & 0 & 0 \\
0 & 0 & 0 & 0 & 0 & 0 & 0 & 0 & 0 & 0 & 0 & 0 & 0 & 0 & 1 & 0 & 0 & 0 & 0 & 0 & 0 & 0 & 0 & 0 & 0 & 0 & 0 \\
0 & 0 & 0 & 0 & 1 & 0 & 0 & 0 & 0 & 0 & 0 & 0 & 0 & 0 & 0 & 0 & 0 & 0 & 0 & 0 & 0 & 0 & 0 & 0 & 0 & 0 & 0 \\
0 & 0 & 0 & 0 & 0 & 0 & 0 & 0 & 0 & 0 & 0 & 0 & 0 & 0 & 0 & 0 & 0 & 1 & 0 & 0 & 0 & 0 & 0 & 0 & 0 & 0 & 0 \\
0 & 0 & 0 & 0 & 0 & 0 & 0 & 1 & 0 & 0 & 0 & 0 & 0 & 0 & 0 & 0 & 0 & 0 & 0 & 0 & 0 & 0 & 0 & 0 & 0 & 0 & 0 \\
0 & 0 & 0 & 0 & 0 & 0 & 0 & 0 & 0 & 0 & 0 & 0 & 0 & 0 & 0 & 0 & 0 & 0 & 0 & 0 & 0 & 0 & 0 & 0 & 1 & 0 & 0 \\
0 & 0 & 0 & 0 & 0 & 0 & 0 & 0 & 0 & 0 & 0 & 0 & 0 & 0 & 0 & 0 & 0 & 0 & 0 & 0 & 0 & 0 & 0 & 1 & 0 & 0 & 0 \\
0 & 0 & 0 & 0 & 0 & 0 & 0 & 0 & 0 & 0 & 0 & 0 & 0 & 1 & 0 & 0 & 0 & 0 & 0 & 0 & 0 & 0 & 0 & 0 & 0 & 0 & 0 \\
0 & 0 & 0 & 1 & 0 & 0 & 0 & 0 & 0 & 0 & 0 & 0 & 0 & 0 & 0 & 0 & 0 & 0 & 0 & 0 & 0 & 0 & 0 & 0 & 0 & 0 & 0 \\
0 & 0 & 1 & 0 & 0 & 0 & 0 & 0 & 0 & 0 & 0 & 0 & 0 & 0 & 0 & 0 & 0 & 0 & 0 & 0 & 0 & 0 & 0 & 0 & 0 & 0 & 0 \\
0 & 0 & 0 & 0 & 0 & 0 & 0 & 0 & 0 & 0 & 0 & 0 & 0 & 0 & 0 & 0 & 0 & 0 & 0 & 1 & 0 & 0 & 0 & 0 & 0 & 0 & 0 \\
0 & 0 & 0 & 0 & 0 & 0 & 0 & 0 & 0 & 1 & 0 & 0 & 0 & 0 & 0 & 0 & 0 & 0 & 0 & 0 & 0 & 0 & 0 & 0 & 0 & 0 & 0 \\
0 & 0 & 0 & 0 & 0 & 0 & 0 & 0 & 0 & 0 & 0 & 0 & 0 & 0 & 0 & 0 & 0 & 0 & 0 & 0 & 0 & 0 & 1 & 0 & 0 & 0 & 0 \\
0 & 0 & 0 & 0 & 0 & 0 & 0 & 0 & 0 & 0 & 0 & 0 & 1 & 0 & 0 & 0 & 0 & 0 & 0 & 0 & 0 & 0 & 0 & 0 & 0 & 0 & 0 \\
0 & 0 & 0 & 0 & 0 & 1 & 0 & 0 & 0 & 0 & 0 & 0 & 0 & 0 & 0 & 0 & 0 & 0 & 0 & 0 & 0 & 0 & 0 & 0 & 0 & 0 & 0 \\
0 & 1 & 0 & 0 & 0 & 0 & 0 & 0 & 0 & 0 & 0 & 0 & 0 & 0 & 0 & 0 & 0 & 0 & 0 & 0 & 0 & 0 & 0 & 0 & 0 & 0 & 0 \\
0 & 0 & 0 & 0 & 0 & 0 & 0 & 0 & 0 & 0 & 0 & 0 & 0 & 0 & 0 & 0 & 0 & 0 & 1 & 0 & 0 & 0 & 0 & 0 & 0 & 0 & 0 \\
0 & 0 & 0 & 0 & 0 & 0 & 0 & 0 & 0 & 0 & 0 & 1 & 0 & 0 & 0 & 0 & 0 & 0 & 0 & 0 & 0 & 0 & 0 & 0 & 0 & 0 & 0 \\
0 & 0 & 0 & 0 & 0 & 0 & 0 & 0 & 0 & 0 & 0 & 0 & 0 & 0 & 0 & 0 & 1 & 0 & 0 & 0 & 0 & 0 & 0 & 0 & 0 & 0 & 0 \\
0 & 0 & 0 & 0 & 0 & 0 & 1 & 0 & 0 & 0 & 0 & 0 & 0 & 0 & 0 & 0 & 0 & 0 & 0 & 0 & 0 & 0 & 0 & 0 & 0 & 0 & 0 \\
0 & 0 & 0 & 0 & 0 & 0 & 0 & 0 & 0 & 0 & 0 & 0 & 0 & 0 & 0 & 0 & 0 & 0 & 0 & 0 & 0 & 0 & 0 & 0 & 0 & 0 & 1 \\
\end{bsmallmatrix}}

\makeatletter
\def\namedlabel#1#2{\begingroup
 #2%
 \def\@currentlabel{#2}%
 \phantomsection\label{#1}\endgroup
}
\makeatother

\usepackage[lite]{amsrefs}

\renewcommand{\PrintDOI}[1]{\href{http://dx.doi.org/\detokenize{#1}}{doi: \detokenize{#1}}%
	\IfEmptyBibField{pages}{, (to appear in print)}{}}

\theoremstyle{plain}
\newtheorem*{theorem*}{Theorem}
\newtheorem*{thmex*}{Theorem~\ref{example}}
\newtheorem*{thmasymp*}{Theorem~\ref{thmAsymp}}
\newtheorem{theorem}{Theorem}[section]

\newtheorem{proposition}[theorem]{Proposition}
\newtheorem{remark}[theorem]{Remark}

\newtheorem{example}[theorem]{Example}

\theoremstyle{definition}
\newtheorem{definition}[theorem]{Definition}

\newcommand{\ben}{\begin{enumerate}}
\newcommand{\een}{\end{enumerate}}

\newcommand{\ed}{\end{document}}

\definecolor{rrr}{rgb}{.9,0,.1}

\definecolor{rr}{rgb}{.8,0,.3}

\newcommand{\RomanNumeralCaps}[1]{\MakeUppercase{\romannumeral #1}}

\usepackage{pdfsync}
\graphicspath{ {./images/} }

\title[Pointed Racks ]{Pointed Racks and Their Applications to Braid Theory}

\author[A. Apollos]{Angel Apollos}
\address{Department of Mathematics, Louisiana State University,
Baton Rouge, LA , USA}
\email{aapoll2@lsu.edu}
\author[J. Ceniceros]{Jose Ceniceros}
\address{Mathematics and Statistics Department, Hamilton College, Clinton, NY, USA}
\email{jcenicer@hamilton.edu}

\begin{document}

\maketitle

\begin{abstract}
We introduce a new algebraic structure called a \emph{pointed rack} and use it to explicitly construct linear representations of the braid group $B_n$. By tracking the propagation of colorings through braid crossings, we define a rack counting matrix that acts as a set-theoretic permutation representation on the space of rack colorings. We prove that this matrix invariant not only recovers the classical rack coloring invariant of the braid's closure but also captures strictly finer off-diagonal structural information of the braid. We demonstrate this by distinguishing braids whose closures have the same classical rack invariant.
\end{abstract}

\parbox{5.5in} {\textsc{Keywords:} quandles, racks, pointed quandles, pointed racks, braids

                \smallskip
                
                }

\section{Introduction}\label{intro}

\emph{Racks}, which are algebraic structures equipped with a self-distributive binary operation satisfying an invertibility condition, were introduced by Fenn and Rourke in 1992 \cite{FR}. These structures have garnered sustained interest from topologists due to their natural application to the study of framed knots and links: the defining axioms of racks correspond precisely to the type II and III Reidemeister moves. Racks are generalizations of \emph{quandles}, algebraic structures that satisfy the rack axioms along with an additional idempotency condition. Quandles were introduced independently by Joyce and Matveev in 1982 \cite{Joyce, Matveev}, and their axioms correspond to all three Reidemeister moves, making them an ideal algebraic structure for studying oriented link equivalence.   

Since the introduction of racks, much work has been devoted to defining computable invariants of oriented framed links derived from racks. These invariants include integer-valued counting invariants, such as the \emph{rack counting invariant} of oriented framed links \cite{N4}. Additionally, a homology and cohomology theory associated with racks was introduced in \cite{FRS}, and cocycle enhancements of the rack counting invariant were defined in \cite{EN1}. Beyond their topological applications, racks are also of intrinsic algebraic interest due to their role in the classification of finite-dimensional pointed Hopf algebras and their relation to set-theoretic solutions of the Yang–Baxter equation \cite{AG,D}. 

Furthermore, there is a deep connection between racks and braid actions \cite{BE}. Emil Artin introduced braids in \cite{A}. While braids are commonly studied via their closures (knots and links) using Alexander's Theorem, they are intrinsically interesting as topological objects. However, extending classical rack coloring invariants to open braids presents a challenge. Because the rack operation is deterministic and bijective at each crossing, coloring the top arcs of an $n$-braid uniquely determines the bottom arcs. Consequently, the raw number of colorings of an open $n$-braid is always $|X|^n$, thus making the rack counting invariant of a braid trivial.

To extract meaningful information from the braids, one can study the Yang-Baxter operator. For links, as demonstrated by Gra\~{n}a \cite{G1}, knot invariants are obtained by taking the trace of this operator. While the trace is a powerful tool for links, it fundamentally collapses the linear operator into a single integer, destroying the rich off-diagonal information that tracks how individual colorings propagate from the top to the bottom of the braid. 

In this paper, we introduce a new algebraic structure called a \emph{pointed rack} to explicitly construct a linear representation of the braid group $B_n$. Rather than collapsing the operator via the trace, we formalize the exact set-theoretic permutation representation of $B_n$ acting on the space of rack colorings by defining a rack counting matrix. We prove that preserving the full matrix structure retains strictly finer information than the trace, allowing us to successfully detect topological differences of braids whose closures have identical classical rack invariants. 

This paper is structured as follows. In Section 2, we review the definition of racks and basic examples. Section 3 reviews the basic definitions and notation of braids. Section 4 introduces pointed racks and the fundamental pointed rack associated with an open braid. In Section 5, we define the pointed rack counting invariant and construct the rack counting matrix, proving that this mapping defines a formal linear representation of the braid group. We also formalize its relationship to classical closed link invariants. Finally, in Section 6, we present explicit computational examples of braids distinguished by our matrices, and in Section 7, we conclude with a discussion of future directions and applications.

\section{Racks}\label{Racks}
In this section, we will review the definition of algebraic structures that we will use in this article, as well as a review of their application to knot theory. For more details on racks, see \cite{FR, EN}.

    \begin{definition} \label{rack}
    A \emph{rack} is a set $X$ with a binary operation $\triangleright: X \times X \rightarrow X$ that satisfies 

    \begin{enumerate} 
        \item for all $y \in X$, the map $\beta_y: X \rightarrow X$ defined as $\beta_y(x) = x \triangleright y$ is a bijection, with inverse $\beta_y^{-1}(x)$ denoted by $x \triangleright^{-1}y$, and 
        \item for all $x,y,z \in X$, we have$(x \triangleright y) \triangleright z = (x \triangleright z) \triangleright (y \triangleright z)$
    \end{enumerate}
    
A rack that also satisfies $x \triangleright x  =x $ is called a \emph{quandle}.
\end{definition}
    
To simplify the notation, we use the following convention throughout the paper. If the operation of the rack $(X, \triangleright)$ is clear or the operation is not explicitly needed, we will refer to the rack simply as $X$.   

\begin{example}
 \end{example}\label{Rack Examples}
 The following are standard examples of racks and quandles:
\begin{itemize}
\item  \label{dihedralrack}The set $X = \mathbb{Z} / n\mathbb{Z}$ is a quandle with operation  
\[x \triangleright y \equiv (2y - x)\mod n\]
called a \emph{dihedral quandle}. We will use this quandle throughout this article and we will denote this quandle by $R_n$.
        \item Any group $G$ is a quandle with operation 
        \[ x \triangleright y = yx^{-1}y \]
        called the \emph{core quandle} of $G$.
        \item Any $\mathbb{Z}[t^{\pm},s]/(s^2-(1-t)s)-$module $X$ is a rack with operation 
        \[ x\triangleright y = tx+sy\]
        called a \emph{$(t,s)$-rack}.
    \end{itemize}

\begin{definition}
A \emph{rack homomorphism} between two racks $(X,\triangleright)$ and $(Y, \ast)$ is a map $\phi: X \rightarrow Y$ such that 
\[ \phi(x_1 \triangleright x_2) = \phi(x_1) \ast \phi(x_2).\]
In the case where $\phi$ is bijective, then $\phi$ is a \emph{rack isomorphism}.
\end{definition}

 Recall that a framed link is a link with a choice of framing curve for each component of the link; more specifically, each framing curve is a longitude of a regular neighborhood of each component. The fundamental rack of an oriented framed link was defined in \cite{FR}. Given an oriented framed link $L$ with diagram $D$, the \emph{fundamental rack of $L$}, denoted $\mathcal{FR}(L)$, is the equivalence classes of rack words in generators corresponding to the arcs in the diagram $D$ under the equivalence relations generated by the rack conditions in Definition~\ref{rack} and the relations at each crossing of $D$.  

In \cite{N4}, the fundamental rack of an oriented framed link was used to define the rack counting invariant of framed links.  Given a finite rack $X$ and an oriented framed link $L$ with diagram $D$, the \emph{rack counting invariant} is defined by $\textup{Col}_X(L)=\vert \textup{Hom}(\mathcal{FR}(L), X) \vert$. Each homomorphism $f \in \textup{Hom}(\mathcal{FR}(L), X)$ assigns to each arc $x_i$ of the diagram an element $f(x_i) \in X$, and is referred to as a \emph{$X$-coloring} of $D$.  

In order for $D$ to have a valid $X$-coloring, the coloring rule in Figure~\ref{figs:ColoringRule} must be satisfied at each crossing in $D$.

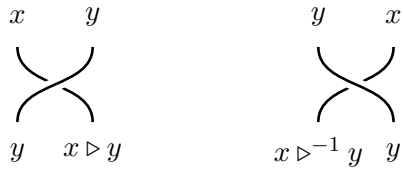
\begin{figure}[H]
\begin{center}
\begin{tikzpicture}
\braid[strands=2,braid start={(-5,0)}]
{\sigma_1^{-1}}
\braid[strands=2,braid start={(-1,0)}]
{\sigma_1}

\node at (-4, .4) {\small$x$}; 
\node at (-3, .4) {\small$y$};
\node at (-4, -1.4) {\small$y$}; 
\node at (-3, -1.4) {\small$x\triangleright y$}; 

\node at (0, .4) {\small$y$}; 
\node at (1, .4) {\small$x$};
\node at (0, -1.4) {\small$x\triangleright^{-1} y$}; 
\node at (1, -1.4) {\small$y$}; 
\end{tikzpicture} 
\end{center}
        \caption{Coloring rules at a positive and negative crossing with downward orientation.}
        \label{figs:ColoringRule}
\end{figure}
Using the above coloring rule, it is a standard exercise to verify that the rack axioms precisely specify the conditions required by the Reidemeister II and III moves to ensure the number of colorings of the diagram $D$ remains the same before and after Reidemeister II and III are applied to $D$, see Figure~\ref{figs:RackAxioms}. Because our focus is braids, we restrict attention to diagrams with downward orientation.

    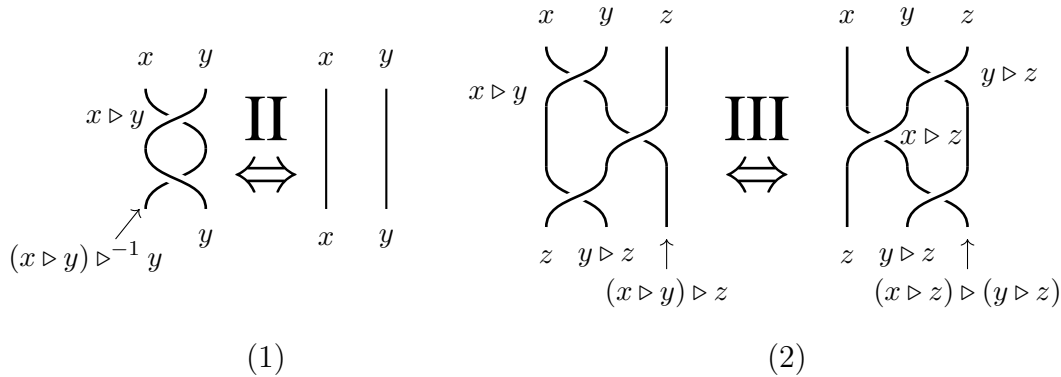
\begin{figure}[H]
    \begin{center}
        \begin{tikzpicture}[scale=.8]
            \braid[strands=2,braid start={(0,0)}]%
            {\sigma_1^{-1}\sigma_1}
            \node[font=\huge] at (3,-.5) {\RomanNumeralCaps{2}};
            \node[font=\Huge] at (3,-1.5) {$\Leftrightarrow$};
            \draw[thick] (4,0) -- (4,-2);
            \draw[thick] (5,0) -- (5,-2);
         
            \node at (1, .5) {\small$x$}; 
            \node at (2, .5) {\small$y$};
            \node at (4, .5) {\small$x$}; 
            \node at (5, .5) {\small$y$};
            \node at (.5, -.5) {\small$x\triangleright y$}; 
            \node at (4, -2.5) {\small$x$}; 
            \node at (5, -2.5) {\small$y$};
            \draw[->] (.5,-2.5) -- (.9,-2);
            \node at (0, -2.8) {\small$(x \triangleright y) \triangleright^{-1}y$};
            \node at (2, -2.5) {\small$y$};
            \node at (3, -4.5) {$(1)$};
        \end{tikzpicture}
        \hspace{.5cm}
   \begin{tikzpicture}[scale=.8]
            \braid[strands=3,braid start={(0,.7)}]%
            {\sigma_1^{-1} \sigma_2^{-1} \sigma_1^{-1}}
            \node[font=\huge] at (4.5,-.5) {\RomanNumeralCaps{3}};
            \node[font=\Huge] at (4.5,-1.5) {$\Leftrightarrow$};
            
            \node at (1, 1.2) {\small$x$}; 
            \node at (2, 1.2) {\small$y$};
            \node at (3, 1.2) {\small$z$};
            \node at (.2, -.1) {\small$x\triangleright y$}; 
            \node at (1, -2.8) {\small$z$}; 
            \node at (2, -2.8) {\small$y \triangleright z$};
            \node at (3, -3.4) {\small$(x \triangleright y) \triangleright z$};
            \draw[->] (3,-3) -- (3,-2.5);

             \braid[strands=3,braid start={(5,.7)}]
            {\sigma_2^{-1} \sigma_1^{-1} \sigma_2^{-1}}
            \node at (6, 1.2) {\small$x$}; 
            \node at (7, 1.2) {\small$y$};
            \node at (8, 1.2) {\small$z$};
            \node at (8.7, .2) {\small$y\triangleright z$}; 
            \node at (7.4, -.8) {$x \triangleright z$};
            \node at (6, -2.8) {\small$z$}; 
            \node at (7, -2.8) {\small$y\triangleright z$};
            \node at (8, -3.4) {\small$(x \triangleright z) \triangleright (y \triangleright z)$};
            \draw[->] (8,-3.0) -- (8,-2.5);
            \node at (5, -4.5) {$(2)$};
        \end{tikzpicture}
    \end{center}
        \caption{Diagrams representing the rack axioms.}
        \label{figs:RackAxioms}
\end{figure}

\begin{example} \label{trefoil}
        Consider the oriented framed trefoil knot $L$ with diagram $D$, see Figure~\ref{fig:FTrefoil}.
\begin{figure}
        \begin{center}
            \begin{tikzpicture}
        \begin{knot}[
          consider self intersections=true,
          flip crossing={3},
            flip crossing={1},
            ]
        
        \strand[<-]
          (0,2) .. controls +(2.2,0) and +(120:-2.2) .. (210:2)
        .. controls +(120:2.2) and +(60:2.2) .. (-30:2) 
        .. controls +(60:-2.2) and +(-2.2,0) .. (0,2);
        
        \end{knot}
        \node at (-2,0) {$x$};
        \node at (0,2.4) {$y$};
        \node at (2,0) {$z$};
        
        \end{tikzpicture}
        \end{center}
        \caption{A diagram $D$ of $L$.}
        \label{fig:FTrefoil}
\end{figure}
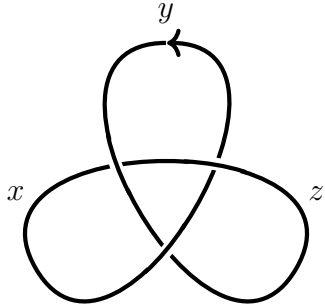
       Using the rack coloring rule in Figure~\ref{figs:ColoringRule}, we obtain the following coloring equations: $y = z \triangleright x$, $x = y \triangleright z$, and $z = x \triangleright y$. Thus,
        \begin{eqnarray*}
            \mathcal{FR}(L)&\cong&\langle x,y,z \,\vert\, y = z \triangleright x, x = y \triangleright z, z = x \triangleright y\rangle\\
            &\cong&\langle x,y \,\vert\, y = (x \triangleright y) \triangleright x, x = y \triangleright (x \triangleright y)\rangle.
        \end{eqnarray*} 
        Using the dihedral quandle $R_3$ the oriented framed link $L$ has 9 valid colorings by $R_3$, each of which we can identify as a tuple $(f(x), f(y))$ where $f \in \textup{Hom}(\mathcal{FR}(L),R_3)$. Hence,
        \[ \textup{Hom}(\mathcal{FR}(L),R_3) = \lbrace (1, 1),(2, 2),(0, 0),(1, 0),(1, 2),(2, 0),(2, 1),(0, 2),
            (0, 1) \rbrace.\]
        Therefore, the rack coloring invariant of $L$ by $R_3$ is $9$.
        
    \end{example}

\section{Braids}\label{Braids}
Having established the foundational properties of racks, we now transition to reviewing some basic facts about braids and braid groups. For more details on braids, see \cite{A1,A,M}.  Braids can be defined in several ways, but in this article, we will briefly review the definitions relevant to this article. We begin with the following definition.

Geometrically, an \emph{$n$-braid} is a collection of $n$ disjoint, properly embedded arcs in a cube, such that each arc begins on the top face, ends on the bottom face, and runs monotonically from the top of the cube to the bottom of the cube. The arcs connect $n$ distinct points on the top face to $n$ distinct points on the bottom face without intersecting. In the figures, we depict the cube as two parallel bars representing the top and bottom faces (see Figure~\ref{fig:Braid}). We do not indicate orientation in the figures; throughout, all braids are understood to be oriented downward, from top to bottom.

           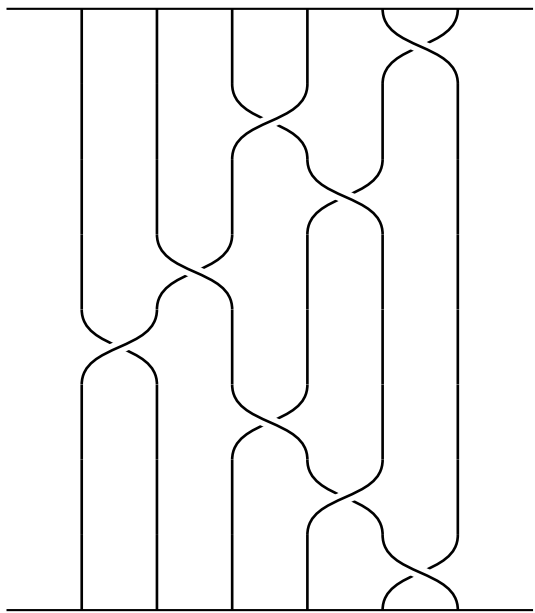
\begin{figure}[H]
               \centering
        \begin{tikzpicture}[scale=.5]
            \braid[strands=6,braid start={(0,0)}]%
        {\sigma_5\sigma_3^{-1}\sigma_4\sigma_2\sigma_1^{-1}\sigma_3\sigma_4^{-1}\sigma_5}
        \draw[thick] (0,0) -- (7,0);    
        \draw[thick] (0,-8) -- (7,-8);
        \end{tikzpicture} 
               \caption{A six-braid.}
               \label{fig:Braid}
           \end{figure}

 Two braids are \emph{isotopic} through any isotopy move of the arcs of the braids that preserves the braid structure. Furthermore, we say that two $n$-braids are \emph{equivalent} or \emph{equal} if they are isotopic. Just like in the case of knots and links, we can obtain a regular diagram of a braid by projecting the braid onto a plane. We will only consider oriented braids and oriented braid diagrams. We will consider the convention that involves braids oriented from top to bottom (downward orientation). In terms of braid diagrams, braid isotopies consist of Reidemeister moves II and III with downward orientation (called braid Reidemeister moves), as well as planar isotopies that preserve monotonicity and endpoints. The planar isotopies allow for the movement of two nonadjacent crossings.  Additionally, an $n$-braid is the equivalence class of braid diagrams under the equivalence relation generated by the braid Reidemeister moves and planar isotopies. 
    
Alternatively, a braid can be defined as a finite product of \emph{generators} $\sigma_i$ and $\sigma_{i}^{-1}$, see Figure~\ref{fig:gen}. Observe that $\sigma_{i}$ represents the $i^{th}$ strand crossing under the $i + 1^{st}$ strand. Alternately, $\sigma_{i}^{-1}$ represents the $i^{th}$ strand crossing over the $i + 1^{st}$ strand. Thus, a braid can be written as a list of crossings from top to bottom, which we call a \emph{word} for the braid.

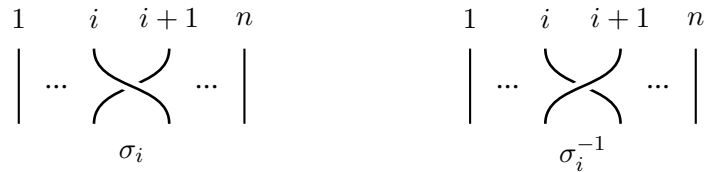
\begin{figure}[H]
     \begin{center}
    \begin{tikzpicture}
    \node at (-4, .4) {\small$1$}; 
    \draw[thick] (-4,0) -- (-4, -1);
    \node at (-3.5, -.5) {...};
    \braid[strands=2,braid start={(-4,0)}]%
    {\sigma_1^{-1}}
     \node at (-1.5, -.5) {...};
     \node at (-1, .4) {\small$n$}; 
    \draw[thick] (- 1,0) -- (-1, -1);
    \node at (-3, .4) {\small$i$}; 
    \node at (-2, .4) {\small$i + 1$};
    \node at (-2.5, -1.4) {\small$\sigma_{i}$};

    \draw[thick] (2,0) -- (2, -1);
    \node at (2, .4) {\small$1$}; 
    \node at (2.5, -.5) {...};
    \braid[strands=2,braid start={(2,0)}]
    {\sigma_1}
    \node at (4.5, -.5) {...};
    \draw[thick] (5,0) -- (5, -1);
    \node at (5, .4) {\small$n$}; 
    
    \node at (3, .4) {\small$i$}; 
    \node at (4, .4) {\small$i + 1$};
    \node at (3.5, -1.4) {\small$\sigma_{i}^{-1}$};
    \end{tikzpicture} 
     \end{center}
     \caption{The generators $\sigma_i$ and $\sigma_i^{-1}$.}
     \label{fig:gen}
     \end{figure}

 \begin{example}
In this example, we have two ways of representing a 3-braid. In Figure~\ref{fig:diaword}, we have a braid diagram and the braid word of a 3-braid. 

\begin{figure}[H]
\begin{center}
           \begin{tikzpicture}
             \braid[strands=4,braid start={(0,0)}]%
             {\sigma_1^{-1}\sigma_2\sigma_1\sigma_3^{-1}\sigma_2^{-1}}
             \node[font=\Huge] at (5,-2.5) {\(=\)};
             \node at (7, -2.5) {\small$\sigma_1\sigma_2^{-1}\sigma_1^{-1}\sigma_3\sigma_2$};
            \end{tikzpicture} 
\end{center}
  \caption{A braid diagram and braid word of a 3-braid.}
    \label{fig:diaword}
\end{figure}
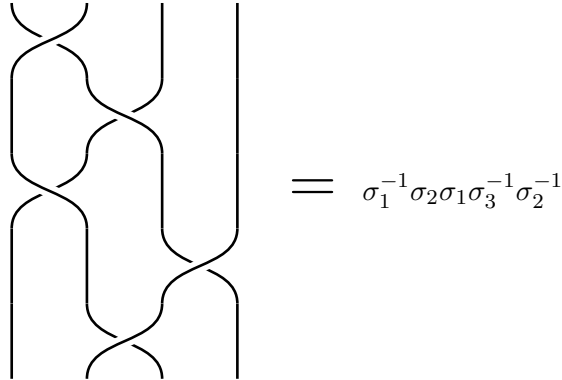
\end{example}

The following group was first introduced by Artin in 1925~\cite{A,A1}. The \emph{$n$-braid group}, denoted by $B_n$, is the set of all $n$-braids under the operation of multiplication, where the product of two $n$-braids $\alpha$ and $\beta$ is defined as the vertical juxtaposition of $\alpha$ followed by $\beta$. This operation is denoted by concatenation: $\alpha\beta$.

The braid group has two \emph{fundamental relations}:
    
    \begin{center}
        \begin{enumerate}
         \item $\sigma_{i}\sigma_{j} = \sigma_{j}\sigma_{i} \qquad\hbox{}\qquad\hbox{}\qquad\hbox{}|i -j| \geq 2$;
         \item $\sigma_{i}\sigma_{i + 1}\sigma_{i} = \sigma_{i + 1}\sigma_{i}\sigma_{i + 1} \qquad\hbox{} i = 1, 2,..., n-2.$
     \end{enumerate} 
    \end{center}

The fundamental relation $(2)$ encodes the Reidemeister III move with the downward orientation, while the fundamental condition $(1)$ encodes the planar isotopy discussed earlier of moving two non-adjacent crossings. Lastly, the Reidemeister II move is automatically captured by the fact that each $\sigma_i$ is invertible. 
    
Since the rack conditions axiomatize the Reidemeister moves II and III, it is tempting to use racks to define a rack counting invariant of braids, but we must be cautious as noted in \cite{CN2, R}. Let $X$ be a finite rack and let $B$ be the diagram of an oriented $n$-braid. At any crossing in $B$, the input colors at a crossing uniquely determine the output color, as shown in Figure~\ref{fig:xcolors}.
\begin{figure}[H]
    \centering
    \begin{center}
\begin{tikzpicture}
\braid[strands=2,braid start={(-5,0)}]
{\sigma_1^{-1}}
\braid[strands=2,braid start={(-1,0)}]
{\sigma_1}

\node at (-4, .4) {\small$x$}; 
\node at (-3, .4) {\small$y$};
\node at (-4, -1.4) {\small$y$}; 
\node at (-3, -1.4) {\small$x\triangleright y$}; 

\node at (0, .4) {\small$y$}; 
\node at (1, .4) {\small$x$};
\node at (0, -1.4) {\small$x\triangleright^{-1} y$}; 
\node at (1, -1.4) {\small$y$}; 
\end{tikzpicture} 

\end{center}
    \caption{Colorings at positive and negative crossings.}
    \label{fig:xcolors}
\end{figure}
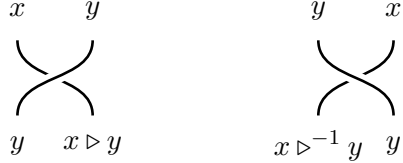
Since we are not considering the closure of the braid, there are no additional restrictions. This means that each selection of colors $x_1, x_2, \dots, x_n \in X$ for the top of $B$ will uniquely determine the colorings of the arcs in the diagram $B$. Thus, the rack counting invariant of an oriented $n$-braid with respect to $X$ is $\vert X \vert^n$. Furthermore, if we consider an $n$-braid $B$ without performing a closure, then its fundamental rack, which we will denote by $\mathcal{R}(B)$, is the free rack on generators $x_1, \dots, x_n$, corresponding to the colors of the top arcs of $B$. Specifically, $\mathcal{R}(B) = \langle x_1, \dots, x_n \rangle$. By the definition of a rack, $\mathcal{R}(B)$ is an invariant of $n$-braids, although a trivial one.

In the following section, we will address this issue and introduce a more effective invariant of braids.

\section{Introduction of pointed racks}\label{Pointed Rack}
In this section, we introduce racks with basepoints. This will allow us to remember the colors at the top and bottom of a braid diagram. We will then use this algebraic structure to define effective invariants of braids. 

\begin{definition}\label{pointedrackdef}
An \emph{$m$-pointed rack} $(X, x_1, \dots, x_m)$ is an ordered $m+1$-tuple where $X$ is a rack and $x_1, x_2, \dots, x_m \in X$. We will call the $x_1, \dots, x_m $ the \emph{basepoints} of the pointed $m$-rack. 
\end{definition}
When $X$ is restricted to a quandle, we obtain a pointed quandle as introduced in \cite{GuPf}. To simplify our notation, we will denote an $m$-pointed rack $(X, x_1, \dots, x_m)$ with the calligraphic letter $\mathcal{X}$.

\begin{definition}
    An \emph{$m$-pointed rack homomorphism} between two $m$-pointed racks $\mathcal{X}=(X,x_1, \dots, x_m)$ and $\mathcal{Y}=(Y,y_1, \dots y_m)$ is a map $\phi: X\rightarrow Y$ that is a rack homomorphism such that $\phi(x_i)= y_i$ for all $i \in \lbrace 1,\dots, m \rbrace$. The set of all homomorphisms between $\mathcal{X}$ and $\mathcal{Y}$ will be denoted by $\textup{Hom}(\mathcal{X}, \mathcal{Y})$. In the case where $\phi$ is an $m$-pointed homomorphism that is a bijection, then $\phi$ is an \emph{$m$-pointed rack isomorphism}.
\end{definition}

\begin{proposition}
    Let $X$ and $Y$ be racks, and let $\mathcal{X}=(X, x_1, \dots, x_m)$ be an $m$-pointed rack. Then $\phi$ is a rack isomorphism between $X$ and $Y$ if and only if the induced map 
    \[\phi: \mathcal{X} \to \phi(\mathcal{X}), \quad \text{where } \phi(\mathcal{X}) = (Y, \phi(x_1), \dots, \phi(x_m)),
\]
is an isomorphism of pointed racks.
\end{proposition}

\begin{proof}

Suppose $\phi: X \to Y$ is a rack isomorphism. By definition, $\phi$ is a bijection, so $\phi(X)=Y$. The map $\phi$ also sends each basepoint $x_i$ to $\phi(x_i)$, so $\phi : \mathcal{X} \to \phi(\mathcal{X})$ defined by $\phi(\mathcal{X}) := (Y, \phi(x_1), \dots, \phi(x_m))$ is a pointed rack isomorphism.

Conversely, suppose $\phi : \mathcal{X} \to \phi(\mathcal{X})$ is a pointed rack isomorphism, then  $\phi : X \to Y$ is a rack isomorphism and maps each $x_i$ to $\phi(x_i)$. Therefore, $\phi$ is a rack isomorphism.
\end{proof}

\begin{definition}
Let $B$ be an $n$-braid. The \emph{fundamental pointed rack of $B$} is the $2n$-pointed rack  
\[
\mathcal{PR}(B) := (\mathcal{R}(B), t_1, \dots, t_n, b_1, \dots, b_n),
\]  
where $\mathcal{R}(B)$ is the fundamental rack of $B$, generated by the labels $t_1, \dots, t_n$ assigned to the top of the arcs of the braid. The labels $b_1, \dots, b_n \in \mathcal{R}(B)$ are the corresponding labels assigned to the bottoms of the arcs of the braid, determined by the braid action and rack operation.

\end{definition}
\begin{remark}
    We emphasize that $\mathcal{R}(B)$ is isomorphic to the free rack generated by $t_1, \dots, t_n$, but the pointed fundamental rack of $B$ has $\mathcal{R}(B)$ as the underlying rack with fixed basedpoint $t_1,\dots t_n, b_1, \dots, b_n$. 
\end{remark}

As in the case of the fundamental rack of an oriented framed link, we can obtain a presentation of the fundamental pointed rack of $B$ by labeling the arcs of a diagram $D$ of the braid $B$. As mentioned earlier, the labels assigned to the top arcs of $B$ will uniquely determine the remaining labels for the arcs of the diagram of $D$. Thus, the top labels are the generators of the fundamental rack of $B$. Furthermore,  if the labels used for the top arcs of $D$ are $t_1,\dots, t_n$ then the bottom labels $b_1, \dots, b_n$ will be rack words determined by $t_1, \dots, t_n$, the braid, and the rack coloring rule.

\begin{theorem}\label{fundamental}
    The fundamental pointed rack is an $n$-braid invariant. 
\end{theorem}

\begin{proof}
Let $B$ be an $n$-braid with diagram $D$. The fundamental rack $ \mathcal{R}(B) $ is the free rack on generators $ t_1, \dots, t_n$. Since we are considering braids and not the closure of the braid, the labels $t_1, \dots, t_n$ of the top arcs of $D$ are generators for a presentation of the fundamental rack of $B$. Furthermore, the labels on all other arcs, including the bottom strands $ b_1, \dots, b_n $, are determined uniquely by $t_1, \dots t_n$, the braid, and the rack coloring rule.

The Reidemeister moves are only applied to a braid diagram away from the endpoints, so the top generators $ t_i $ are unchanged, and each bottom label $ b_i \in \mathcal{R}(B)$ is the image of $ t_i \in \mathcal{R}(B) $ under the braid. This image is determined by the same word in the free rack, up to equivalence under the rack axioms, which correspond to the Reidemeister moves. So the rack isomorphism between fundamental racks before and after the Reidemeister move maps the endpoint labels to endpoint labels, hence it is a pointed rack isomorphism. It follows that the fundamental pointed rack is a braid invariant.
\end{proof}

\begin{figure}[H]
    \centering
    \begin{center}
\begin{tikzpicture}
\braid[strands=2,braid start={(-5,-1)}]
{\sigma_1^{-1}\sigma_1^{-1}}
\braid[strands=2,braid start={(-1,0)}]
{\sigma_1^{-1}\sigma_1^{-1}\sigma_1^{-1}}

\node at (-3.5, 0) {\small$B_1$}; 
\node at (.5, 1) {\small$B_2$};

\node at (-4, -.6) {\small$x$}; 
\node at (-3, -.6) {\small$y$};
\node at (-4, -3.4) {\small$x \triangleright y$}; 
\node at (- 2.5, -3.4) {\small$ y \triangleright (x\triangleright y)$}; 

\node at (0, .4) {\small$x$}; 
\node at (1, .4) {\small$y$};
\node at (0, -3.4) {\small$ y \triangleright (x\triangleright y)$}; 
\node at (3, -3.8) {\small$(x \triangleright y) \triangleright (y \triangleright (x\triangleright y)) $}; 
\draw[->] (3,-3.5) -- (1.2,-3.1);
\end{tikzpicture} 

\end{center}
    \caption{Braids $B_1$ and $B_2$}
    \label{fig:FundEx}
\end{figure}
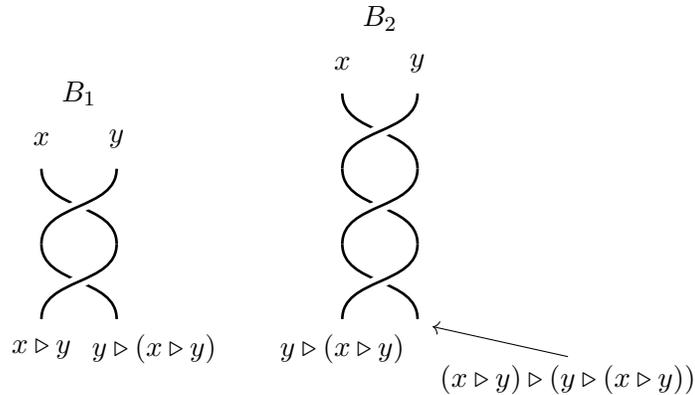

\begin{example}
    Consider the diagrams of the $2$-braids $B_1$ and $B_2$ in Figure~\ref{fig:FundEx}. Their fundamental racks are: 
    \[\mathcal{R}(B_1)  \cong \langle x, y \rangle \cong \mathcal{R}(B_2).\]
    Their fundamental pointed racks are:
    \[ \mathcal{PR}(B_1)  \cong (\mathcal{R}(B_1), x, y, x \triangleright y, y \triangleright (x\triangleright y) )\] and 
    \[ \mathcal{PR}(B_2)  \cong (\mathcal{R}(B_2), x, y, y \triangleright (x\triangleright y), (x \triangleright y) \triangleright (y \triangleright (x\triangleright y)) ).\]
Let $X = R_3$ be the dihedral quandle of order 3 and $\mathcal{X} = (X, 0,1,2,0)$. We can construct a pointed homomorphism $\phi: \mathcal{PR}(B_1) \to \mathcal{X}$ by defining $\phi$ on the generators:
\[ \phi(x) = 0, \quad \phi(y) =1.\]
Next we compute the following: 
\begin{eqnarray}
  \phi(x\triangleright y)&=& \phi(x) \triangleright \phi(y) = 0 \triangleright 1 = 2,  \\
  \phi(y \triangleright ( x \triangleright y)) &=& \phi(y) \triangleright ( \phi(x) \triangleright \phi(y)))=1 \triangleright ( 0 \triangleright 1) =  0.
\end{eqnarray}
So the map $\phi$ sends the basepoints of $\mathcal{PR}(B_1)$ to the basepoints of $\mathcal{X}$. Therefore, $\phi$ is a pointed rack homomorphism.

On the other hand, let $\phi: \mathcal{PR}(B_2) \to \mathcal{X}$ be defined on the generators by 
\[ \phi(x) =0, \quad \phi(y)=1. \] 
Then, 
\begin{eqnarray}
\phi(y \triangleright ( x \triangleright y)) &=& \phi(y) \triangleright ( \phi(x) \triangleright \phi(y)))=1 \triangleright ( 0 \triangleright 1) = 0,\\
\phi((x \triangleright y) \triangleright (y \triangleright ( x \triangleright y))) &=& (\phi(x) \triangleright \phi(y)) \triangleright (\phi(y) \triangleright ( \phi(x) \triangleright \phi(y)))) = 1.
\end{eqnarray}
The image of the basepoints of $\mathcal{PR}(B_2)$ is $(0,1,0,1)$, which are not the basepoints of $\mathcal{X}$. Thus, there does not exist a pointed homomorphism $\phi: \mathcal{PR}(B_2) \rightarrow \mathcal{X}$. 

We see that although the fundamental racks of these two braids are isomorphic, their pointed fundamental racks are not isomorphic. Therefore, the pointed fundamental rack can be used to define effective invariants of braids. 
\end{example}

\section{The Pointed Rack Counting Invariant} \label{Counting Invariant}

In this section, we introduce an integer-value invariant of braids and a matrix-value invariant. 
For any $n$-braid $B$, there is an associated fundamental pointed rack, $\mathcal{PR}(B)$. Let $(X, x_1, x_2, \dots, x_{2n})$ be a finite pointed rack. The set of pointed rack homomorphisms from the fundamental pointed rack of $B$ and the finite pointed rack, $\textup{Hom}(\mathcal{PR}(B), X)$, can be used to define the following integer-invariant. 

\begin{definition} 
    Let $B$ be an $n$-braid and $\mathcal{X}=(X, x_1,\dots, x_{2n})$ be a $2n$-pointed rack. The \emph{pointed rack counting invariant of $B$} is
    \[ Col_\mathcal{X}(B) = \vert \textup{Hom}(\mathcal{PR}(B), \mathcal{X}) \vert.  \]
\end{definition}

\begin{theorem}\label{inv}
    The pointed rack counting invariant is an $n$-braid invariant.
\end{theorem}

\begin{proof}
    This follows from Theorem~\ref{fundamental}.
\end{proof}
For the readers who are familiar with the quandle counting invariant might find it interesting that the pointed rack counting invariant may have a value of 0. Additionally, we prove that this invariant only takes values of 0 or 1. The result is simpy a byproduct of the deterministic, bijective nature of the rack operation at the crossings of a braid. 
\begin{proposition}\label{deterministic}
Let $B$ be an $n$-braid and let $\mathcal{X}$ be a $2n$-pointed rack. The pointed rack counting invariant of $B$ satisfies
\[
\textup{Col}_{\mathcal{X}}(B) \in \{0, 1\}.
\]
\end{proposition}

\begin{proof}
    Let $B$ be an $n$-braid and let $\mathcal{X}$ be a $2n$-pointed rack. The fundamental pointed rack, $\mathcal{PR}(B)$ is generated by $n$ generators, $t_1,t_2,...,t_n$.  Recall that the bottom arc labels $b_1, b_2,...b_n$ are uniquely determined by the top arc labels $t_1,t_2,...,t_n$, the braid relations, and the rack coloring rule. Thus, for each set of $t_1,t_2,...,t_n$, there is only one set of $b_1,b_2,...,b_n$. A pointed rack homomorphism $\phi \colon \mathcal{PR}(B) \rightarrow \mathcal{X}$ maps the generators to their images in $\mathcal{X}$, preserving the structure of the rack and the basepoints. Because $b_i$ are uniquely determined by $t_i$, the image of $\mathcal{PR}(B)$ under any such homomorphism is uniquely determined by the images of $t_i$.
 Thus, it follows that if the images $\phi(t_i)$ yield $\phi(b_i)$ equal to the corresponding basepoints in $\mathcal{X}$, then this is the only valid rack coloring and $\operatorname{Col}_{\mathcal{X}}(B) = 1$. On the other hand, if this condition fails for any $\phi(b_i)$ then there is no coloring and  $Col_\mathcal{X}(B) = 0.$
\end{proof}

        \begin{example}
        Let $X = R_3$ be the dihedral quandle of order 3 and consider the pointed quandle $\mathcal{X}=(X, 0, 1, 2, 2, 1, 2)$.
        Let $B_1 = \sigma_2\sigma_2\sigma_2\sigma_2\sigma_2\sigma_1\sigma_2^{-1}\sigma_1$ and $B_2 = \sigma_2^{-1}\sigma_2^{-1}\sigma_2^{-1}\sigma_2^{-1}\sigma_1^{-1}\sigma_2\sigma_1^{-1}\sigma_1^{-1}$ be two 3-braids, see Figure~\ref{fig:DiagramsB1B2}.
        There exists a pointed rack homomorphism $\phi : \mathcal{PR}(B_1) \rightarrow \mathcal{X}$ defined by $\phi(t_1)=0, \phi(t_2)=1$, and $\phi(t_3)=2$. On the other hand, there does not exist a pointed homomorphism from $\mathcal{PR}(B_2)$ to $\mathcal{X}$. Thus, $Col_\mathcal{X}(B_1) = 1$ and $Col_\mathcal{X}(B_2) = 0.$  Therefore, since 
        $Col_\mathcal{X}(B_1) \not= Col_\mathcal{X}(B_2)$ the pointed counting invariant distinguishes these two braids. 

        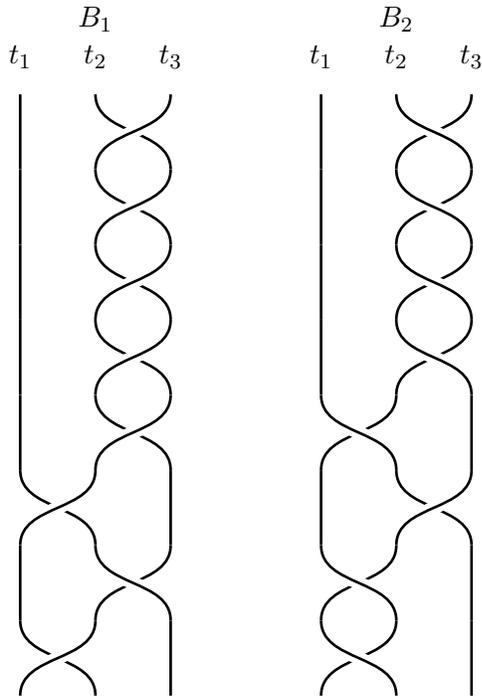
\begin{figure}[H]
        \begin{center}
            \begin{tikzpicture}
        \braid[strands=3,braid start={(-5,0)}]{\sigma_2^{-1}\sigma_2^{-1}\sigma_2^{-1}\sigma_2^{-1}\sigma_2^{-1}\sigma_1^{-1}\sigma_2\sigma_1^{-1}}
        \braid[strands=3,braid start={(-1,0)}]
        {\sigma_2\sigma_2\sigma_2\sigma_2\sigma_1\sigma_2^{-1}\sigma_1\sigma_1}

        \node at (-3, 1) {\small$B_1$}; 
        \node at (1, 1) {\small$B_2$};
        \node at (-4,0.5) {\small$t_1$};
        \node at (-3,0.5) {\small$t_2$};
        \node at (-2,0.5) {\small$t_3$};
        \node at (0,0.5) {\small$t_1$};
        \node at (1,0.5) {\small$t_2$};
        \node at (2,0.5) {\small$t_3$};
        \end{tikzpicture}
        \end{center}
\caption{Braid diagrams of $B_1$ and $B_2$.}
\label{fig:DiagramsB1B2}
    \end{figure}
        \end{example}

\begin{definition}
Let $B$ be an $n$-braid and let $X$ be a finite rack with $\vert X \vert = m$. The \emph{rack counting matrix} of $B$ with respect to $X$, denoted $\Phi_X^{M_{m^n}}(B)$, is the $m^n \times m^n$ matrix whose rows and columns are indexed by $n$-tuples in $X^n$. The $(i,j)$-entry of the matrix is defined by
\[
\left( \Phi_X^{M_{m^n}}(B) \right)_{i,j} =
\begin{cases}
1 & \text{ if the colors of the top arcs indexed by } i \text{ extends to bottom colors indexed } j, \\
0 & \text{otherwise}.
\end{cases}
\]

\end{definition}

\begin{remark} 
By construction, each row index $i \in X^n$ corresponds to an assignment of labels to the top arcs of $B$. Each column index $j \in X^n$ corresponds to a possible assignment of labels to the bottom arcs of $B$. Thus, we will refer to $i$ and $j$ as coloring vectors.
\end{remark}

Proposition~\ref{deterministic} will play an essential role for proving the representation property of the rack counting matrix.

\begin{theorem}\label{linearRep}
    Let $X$ be a finite rack with $\vert X \vert =m$ and let $B_1$ and $B_2$ be two $n$-braids. Then their rack counting matrix of the concatenated braid $B_1B_2$ satisfies
    \[ \Phi_X^{M_{m^n}}(B_1B_2) = \Phi_X^{M_{m^n}}(B_1)\cdot \Phi_X^{M_{m^n}}(B_2). \]
As a result, the mapping $B \mapsto \Phi_X^{M_{m^n}}(B)$ defines a linear representation of the braid group $B_n$.
\end{theorem}
\begin{proof}
    By definition $(\Phi_X^{M_{m^n}}(B))_{i,j} = 1$ if and only if the top coloring vector $i$ can be extended to the bottom coloring vector $j$ through the braid $B_1B_2$. Because colorings propagate uniquely by the braid crossings and rack operations, this occurs if and only if there exists a coloring vector $k$ at the boundary between $B_1$ and $B_2$ such that the coloring vector $i$ extends to $k$ through $B_1$ and then the coloring vector $k$ extends to the coloring vector $j$ through the braid $B_2$. Thus 
    \[(\Phi_X^{M_{m^n}}(B_1))_{ik}=1 \quad \text{and} \quad (\Phi_X^{M_{m^n}}(B_2))_{kj}=1.\]
    Next, we evaluate the corresponding entry for the product of the individual matrices. By the definition of matrix multiplication, we obtain
    \[ (\Phi_X^{M_{m^n}}(B_1)\cdot \Phi_X^{M_{m^n}}(B_2))_{ij} = \sum_{k\in X^n}(\Phi_X^{M_{m^n}}(B_1))_{ik}(\Phi_X^{M_{m^n}}(B_2))_{kj}.\]
    By Proposition~\ref{deterministic}, the propagation of the top coloring of a braid is completely deterministic. Therefore, given a top coloring vector $i$, there exists a unique intermediate coloring vector $k'$ at the bottom of $B_1$. Thus, $(\Phi_X^{M_{m^n}}(B_1))_{ik}=1$ when $k'=k$, and $0$ for all other $k$. Hence,
    \[\sum_{k \in X^n}(\Phi_X^{M_{m^n}}(B_1))_{ik}(\Phi_X^{M_{m^n}}(B_2))_{kj} = (1)(\Phi_X^{M_{m^n}}(B_2))_{k'j} \]
    Thus, the sum equals 1 if and only if the unique intermediate coloring vector extends to the bottom coloring vector $j$ through the braid $B_2$. Therefore, the matrices are equal. We have shown that this mapping preserves the braid group operation by translating the operation of concatenation into matrix multiplication, this defines a linear representation of the braid group $B_n$.
\end{proof}

\begin{theorem}
    Let $X$ be a finite rack, and $B$ an $n$-braid. The rack counting matrix of $B$ is an invariant of $B$.
\end{theorem}
\begin{proof}
    This follows directly from Theorem~\ref{inv}.
\end{proof}
\begin{example}\label{BraidEx}
    Consider the two braids in Figure \ref{fig:FundEx}. Let $X = R_3$ be the dihedral quandle of order 3. 
    Then $$\Phi_X^{M_{3^2}}(B_1)=
    \begin{bmatrix}
    1 & 0 & 0 & 0 & 0 & 0 & 0 & 0 & 0\\
    0 & 0 & 0 & 0 & 0 & 0 & 1 & 0 & 0\\
    0 & 0 & 0 & 1 & 0 & 0 & 0 & 0 & 0\\
    0 & 0 & 0 & 0 & 0 & 0 & 0 & 1 & 0\\
    0 & 0 & 0 & 0 & 1 & 0 & 0 & 0 & 0\\
    0 & 1 & 0 & 0 & 0 & 0 & 0 & 0 & 0\\
    0 & 0 & 0 & 0 & 0 & 1 & 0 & 0 & 0\\
    0 & 0 & 1 & 0 & 0 & 0 & 0 & 0 & 0\\
    0 & 0 & 0 & 0 & 0 & 0 & 0 & 0 & 1
    \end{bmatrix}$$ and 
    $$\Phi_X^{M_{3^2}}(B_2)=
    \begin{bmatrix}
    1 & 0 & 0 & 0 & 0 & 0 & 0 & 0 & 0\\
    0 & 1 & 0 & 0 & 0 & 0 & 0 & 0 & 0\\
    0 & 0 & 1 & 0 & 0 & 0 & 0 & 0 & 0\\
    0 & 0 & 0 & 1 & 0 & 0 & 0 & 0 & 0\\
    0 & 0 & 0 & 0 & 1 & 0 & 0 & 0 & 0\\
    0 & 0 & 0 & 0 & 0 & 1 & 0 & 0 & 0\\
    0 & 0 & 0 & 0 & 0 & 0 & 1 & 0 & 0\\
    0 & 0 & 0 & 0 & 0 & 0 & 0 & 1 & 0\\
    0 & 0 & 0 & 0 & 0 & 0 & 0 & 0 & 1
    \end{bmatrix}.$$
    Observe that $\Phi_X^{M_{3^2}}(B_1) \not= \Phi_X^{M_{3^2}}(B_2)$. Therefore, the $B_1$ and $B_2$ are distinguished by the rack counting matrix invariant.
    
\end{example}

\begin{proposition} \label{trace}
Let $B$ be an $n$-braid and let $X$ be a finite rack. Let $\widehat{B}$ denote the closure of $B$, then the rack counting invariant of $\widehat{B}$ with respect to $X$ is given by 
\[
\textup{Col}_X(\widehat{B})=\vert \textup{Hom}(\mathcal{FR}(\widehat{B}), X) \vert = \operatorname{tr}\left( \Phi_X^{M_{m^n}}(B) \right).
\]
\end{proposition}

\begin{proof}
Each entry $(\Phi_X^{M_{m^n}}(B))_{i,j}$ is $1$ if the coloring of the top arcs of $B$ indexed by $i$ yields the bottom coloring indexed by $j$ under the rack coloring rules. In the closure $\widehat{B}$, each top arc is connected to the corresponding bottom arc, so a valid rack coloring of $\widehat{B}$ with respect to $X$ must have matching top and bottom colors.

Therefore, the number of such colorings is equal to the trace of the rack counting matrix. 
\end{proof}    

\begin{remark}
As established in Proposition~\ref{linearRep}, this mapping is a linear representation of the braid group. Retaining the full matrix preserves the off-diagonal permutation propagation of rack colorings by the braid's action. While taking the trace collapses this representation to recover the classical rack counting invariant for closed links, in alignment with Gra\~{n}a \cite{G1}, the uncollapsed matrix provides combinatorial information about the open braid that is destroyed by the trace, see Table 2.
\end{remark}

\section{Examples}\label{Example} 

In this section, we compute the rack counting matrix for various braids. Specifically, we provide an example of the rack counting invariant of an oriented framed knot, computed using the rack counting matrix. Additionally, we present examples of the rack counting matrix invariant for braids whose closures represent knots with up to five crossings. The computations in this section were done using custom Python code.

\begin{example}
    Let $B = \sigma_1\sigma_1\sigma_1$ be a 2-braid and let $X = R_3$ be the dihedral quandle of order 3. 
    Then $$\Phi_X^{M_{3^2}}(B)=\begin{bmatrix}
    1 & 0 & 0 & 0 & 0 & 0 & 0 & 0 & 0\\
    0 & 1 & 0 & 0 & 0 & 0 & 0 & 0 & 0\\
    0 & 0 & 1 & 0 & 0 & 0 & 0 & 0 & 0\\
    0 & 0 & 0 & 1 & 0 & 0 & 0 & 0 & 0\\
    0 & 0 & 0 & 0 & 1 & 0 & 0 & 0 & 0\\
    0 & 0 & 0 & 0 & 0 & 1 & 0 & 0 & 0\\
    0 & 0 & 0 & 0 & 0 & 0 & 1 & 0 & 0\\
    0 & 0 & 0 & 0 & 0 & 0 & 0 & 1 & 0\\
    0 & 0 & 0 & 0 & 0 & 0 & 0 & 0 & 1
    \end{bmatrix}.$$

    Computing the trace of the matrix counting invariant, we obtain $\textup{Col}_X(\widehat{B})=tr(\Phi_X^{M_{3^2}}(B)) = 9$. The closure of the braid $B$ is the knot in Example~\ref{trefoil}, and we obtain that the rack counting invariant of this knot with respect to $R_3$ is 9. 
\end{example}

\begin{example}
Let $X=R_3$ be the dihedral quandle of order 3. We compute the matrix counting invariant for all braids whose closures represent knots with up to 5 crossings. We collect the results for 2-braids in Table~\ref{table:2braids} and the results for 3-braids in Table~\ref{table:3braids}.

\begin{longtable}{|c|>{\centering\arraybackslash}m{3cm}|c|c|}

     \hline
     \centering
    $\Phi_X^{M_{3^2}}(B)$ & $\textup{tr}(\Phi_X^{M_{3^2}}(B))$ & \textbf{$2$-Braid $B$} & $\widehat{B}$ \\
    \hline
    \endfirsthead

    \raisebox{3pt}[3cm][2.5cm]{$\mOne$} & 9 & \begin{minipage}{3cm}
    \begin{itemize}
    \item $\sigma_1^{-1}\sigma_1$
    \item $\sigma_1^{-1}\sigma_1^{-1}\sigma_1^{-1}$ 
    \end{itemize} \end{minipage}& \begin{minipage}{2cm}
    \begin{itemize}
    \item $0_1$
    \item $3_1$ 
    \end{itemize} \end{minipage}\\
    \hline
    \raisebox{3pt}[3cm][2.5cm]{$\mTwo$} & 3 & $\sigma_1^{-1}\sigma_1^{-1}\sigma_1^{-1}\sigma_1^{-1}\sigma_1^{-1}$ & $5_1$\\
    \hline
       \caption{Invariants for 2-braids}
    \label{table:2braids}
\end{longtable}

\begin{longtable}{|c|>{\centering\arraybackslash}m{3cm}|c|c|}

     \hline
     \centering
    $\Phi_X^{M_{3^3}}(B)$ & $\textup{tr}(\Phi_X^{M_{3^3}}(B))$ & \textbf{$3$-Braid $B$} & $\widehat{B}$ \\
    \hline
    \endfirsthead

    \raisebox{3pt}[4cm][3.5cm]{$\mFour$} 
    & 3 &$\sigma_2^{-1}\sigma_1\sigma_2^{-1}\sigma_1$& $4_1$\\
    \hline
    \raisebox{3pt}[4cm][3.5cm]{$\mFive$} & 3 & $\sigma_2^{-1}\sigma_2^{-1}\sigma_2^{-1}\sigma_1^{-1}\sigma_2\sigma_1^{-1}$& $5_2$ \\
    \hline
       \caption{Invariants for 3-braids}
    \label{table:3braids}
\end{longtable}
\end{example}

\section{Future Directions}

The results of this article open several promising routes for future research.

First, interpreting the rack counting matrix as a formal linear representation of the braid group $B_n$ bridges the abstract braid action with the computational power of linear algebra. Future work may explore the structural decomposition of these matrices to extract further topological data.

Furthermore, because the counting operator acts as a permutation matrix, it inherently defines a directed graph or a \emph{rack coloring quiver}. Investigating the cycle structure of these disjoint quivers could yield interesting combinatorial invariants that track the higher-order periodicity of the braid action. Additionally, defining a pointed rack homology and cohomology theory might provide useful information about braids and other open-ended topological objects. 

Additionally, this unreduced matrix structure provides the exact algebraic foundation required for homological enhancements. Rather than restricting state-sums to closed links, the non-zero entries of our representation matrix could be replaced with polynomial Boltzmann weights derived from rack 2-cocycles.

Finally, the application of pointed racks to inherently open-ended topological objects, such as oriented framed tangles and braidoids, seems like a promising direction that should yield highly effective invariants.
\bibliography{Ref}
\bibliographystyle{plain}

\end{document}